\documentclass[11pt]{article}

\usepackage[a4paper, left=3.5cm,top=3.2cm,right=3.5cm,bottom=3.7cm]{geometry}
\usepackage{graphicx}
\usepackage{mathtools}
\usepackage{psfrag}
\usepackage{amssymb}
\usepackage{amsmath}
\usepackage{authblk}
\usepackage{siunitx}
\usepackage{bm}

\usepackage{amsthm}
\usepackage{hyperref}

\usepackage{todonotes}

\usepackage{algorithm}
\usepackage{algorithmic}
\usepackage{subfig}

\usepackage{xspace}

\usepackage{enumitem}
\setitemize{label=\scriptsize{$\blacksquare$},topsep=0em}
\setenumerate{label=(\alph*), topsep=0em}

\newtheorem{theorem}{Theorem}
\newtheorem{lemma}[theorem]{Lemma}
\newtheorem{proposition}[theorem]{Proposition}
\newtheorem{corollary}[theorem]{Corollary}

\theoremstyle{definition}

\newtheorem{definition}[theorem]{Definition}

\newcommand{\R}{\mathbb R}

\newcommand{\N}{\mathbb N}

\newcommand{\edot}{\,\cdot\,}

\def\plus{{\boldsymbol{\texttt{+}}}}

\DeclareMathOperator*{\argmin}{arg\,min}

\newcommand{\al}{\alpha}
\newcommand{\la}{\lambda}

\newcommand{\ran}{\operatorname{ran}}
\newcommand{\dom}{\operatorname{dom}}

\newcommand\abs[1]{\left\vert#1\right\vert}
\newcommand\sabs[1]{{\lvert#1\rvert}}
\newcommand\norm[1]{{\left\Vert#1\right\Vert}}
\newcommand\snorm[1]{\Vert#1\Vert}

\newcommand\set[1]{{\left\{#1\right\}}}
\newcommand{\kl}[1]{\left(#1\right)}

\newcommand{\skl}[1]{(#1)}
\newcommand{\sset}[1]{\{#1\}}

\newcommand\rest[2]{{#1}\vert_{#2}}

\makeatletter
\newcommand*\bigcdot{\mathpalette\bigcdot@{.6}}
\newcommand*\bigcdot@[2]{\mathbin{\vcenter{\hbox{\scalebox{#2}{$\m@th#1\bullet$}}}}}
\makeatother

\newcommand{\signal}{x}
\newcommand{\zsignal}{z}

\newcommand{\data}{y}
\newcommand{\noise}{\xi}

\newcommand{\XX}{X}
\newcommand{\YY}{Y}
\newcommand{\ZZ}{Z}

\newcommand{\BP}{\mathbf{B}}

\newcommand{\Lo}{\mathbf{L}}

\newcommand{\nlo}{\sigma}              
\newcommand{\Po}{\mathbf  P}                         
\newcommand{\Qo}{\mathbf  Q}                         
\newcommand{\Ao}{\mathbf  A}                         
\newcommand{\Bo}{\mathbf  B}                         
\newcommand{\Ro}{\mathbf R}    

\newcommand{\M}{\mathcal M}                         
\newcommand{\Wo}{\mathbf W}                         
\newcommand{\Wset}{\mathcal{W}}                         
\newcommand{\Nset}{\mathcal{N}}                         
\newcommand{\Id}{\operatorname{Id}}                         

\newcommand{\tik}{\mathcal{T}}     

\newcommand{\NN}{\mathbf{N}}  
\newcommand{\nun}{{\bm \Phi}}  
\newcommand{\nsn}{\mathbf{L}}  

\usepackage{soul}
\usepackage{xcolor}
\colorlet{lred}{red!40}
\colorlet{lgreen}{green!40}
\colorlet{lblue}{blue!40}

\numberwithin{equation}{section}
\numberwithin{figure}{section}
\numberwithin{theorem}{section}

\allowdisplaybreaks

\title{Deep Null Space Learning for  Inverse Problems:
Convergence Analysis and Rates}

\author{Johannes~Schwab}
\author{Stephan~Antholzer}
\author{Markus~Haltmeier$^\star$}

\affil{Department of Mathematics, University of Innsbruck\authorcr
Technikerstrasse 13, 6020 Innsbruck, Austria\authorcr
\vspace{1em}
$^\star$Correspondence: {\tt markus.haltmeier@uibk.ac.at}}

\date{July 22, 2018}

\begin{document}

\maketitle

\begin{abstract}
Recently, deep learning based methods appeared
 as a new paradigm for  solving  inverse problems.
 These methods empirically show excellent
 performance but lack of theoretical justification; in particular, no
  results on the regularization properties  are available.
  In particular, this is  the  case for  two-step deep learning approaches,
   where a classical reconstruction method is applied to the data
   in a first step and  a trained deep
   neural network   is applied  to improve results
   in a second step.
 In this paper, we close the gap between  practice and theory 
 for a new  network structure in a two-step approach.
 For that purpose, we propose so called null space networks and   
 introduce the concept of $\M$-regularization.  Combined  with a 
 standard regularization method as reconstruction layer, the  proposed deep null space learning  
 approach is shown to be a $\M$-regularization method; convergence rates are also 
 derived.
The proposed null space network structure naturally 
preserves  data consistency  which is  considered as key property of
neural networks for solving inverse problems.


\end{abstract}

\section{Introduction}
\label{sec:intro}

We study the solution of inverse problems of the form
\begin{equation}\label{eq:ip}
	\text{Estimate $ \signal \in \XX $ from data } \quad
	\data^\delta = \Ao \signal   + \noise    \,.
\end{equation}
Here $\Ao \colon  \XX \to \YY$  is a  linear  operator
between Hilbert  spaces $\XX$ and $\YY$, and
$\noise  \in \YY$ models the unknown data error (noise),
which is assumed to satisfy the estimate   $\snorm{\noise } \leq \delta $
for some  noise level $\delta  \geq 0$.
We thereby allow a possibly  infinite-dimensional function space  setting,
but clearly the approach and results apply to a finite dimensional setting as well.

We focus on the ill-posed
(or ill-conditioned) case where, without additional  information,
the solution of \eqref{eq:ip} is either highly unstable, highly undetermined, or both.
Many inverse problems in biomedical imaging,
geophysics, engineering sciences, or elsewhere  can be written in such a form (see, for example,  \cite{engl1996regularization,scherzer2009variational}).
For its stable solution one has to employ regularization methods,  which are based on
approximating~\eqref{eq:ip} by neighboring well-posed problems, which enforce stability, accuracy,
 and  uniqueness.

\subsection{Regularization methods}

Any method  for the stable solution of \eqref{eq:ip} uses, either implicitly or explicitly,
a-priori information about  the unknowns to be  recovered. Such
information can be that  $\signal$ belongs to a certain set
of admissible elements $\M$  or that it  has small value of some
regularizing functional.
 The most  basic  regularization method is probably
 Tikhonov regularization, where the solution is defined as a minimizer
 of the quadratic  Tikhonov functional
   \begin{equation}\label{eq:tik}
 \tik_{\alpha;\data_\delta}(\signal)\coloneqq
 \frac{1}{2} \norm{\Ao (\signal) - \data_\delta}^2
 + \frac{\alpha}{2} \norm{ \signal }^2
 \end{equation}
Other classical regularization methods for solving linear inverse
problems are filter based methods \cite{engl1996regularization},
which include Tikhonov regularization as special case.

In the last couple of years  variational  regularization methods
 including   TV regularization or  $\ell^q$ regularization
 became  popular  \cite{scherzer2009variational}.
 They also include classical Tikhonov regularization as special case. 
 In the general version, the regularizer $\frac{1}{2} \norm{ \edot }^2 $    is replaced by
  general convex and lower semi-continuous functionals.

In this paper, we develop a new regularization concept that we name
$\M$-regularization method. Roughly spoken, an $\M$-regularization method
is a tuple $((\Ro_\al)_{\al >0},  \al^\star)$ where (for  a precise definition see  Definition~\ref{def:Mreg})
\begin{itemize}
\item  $\M \subseteq \XX$ is the set of admissible elements;
\item  $\Ro_\al \colon  \YY \to \XX$ are  continuous mappings;
\item  $ \alpha^\star = \alpha^\star(\delta, \data^\delta)$ is a suitable parameter choice;
\item  For any $ \signal \in \M$ we have $\Ro_{\alpha^\star(\delta, \data^\delta)}
(\data^\delta)  \to \signal$  as $\delta \to 0$.
\end{itemize}
Note that for some cases is might be reasonable  to take $\Ro_\al$  multivalued.
For the sake of  simplicity here we only consider the single-valued case.
Classical regularization methods are special cases of $\M$-regularization methods
in Hilbert spaces  where
$\M = \ker(\Ao)^\bot$. A typical regularization method is in this case
given by  Tikhonov regularization, where  $ \Ro_\al = (\Ao^* \Ao + \al \Id_{\XX} )^{-1}
 \Ao^*$.

\subsection{Solving inverse problems by neural networks}

Very recently,  deep learning  approaches  appeared  as alternative,
 very successful methods for solving inverse problems
(see, for example, \cite{adler2017solving,antholzer2017deep,benning2018modern,chang2017one,han2016deep,kobler2017variational,gupta2018cnn,chen2017lowdose,jin2017deep,wang2016accelerating,wang2016perspective,wurfl2016deep,zhang2016image}).
In  most of these approaches, a reconstruction network $\Ro \colon \YY \to \XX$  is trained  to map  measured data to the desired output image.

Various reconstruction networks  have been introduced in the literature.
In the two-step approach, the reconstruction networks take the form 
$\Ro  = \nsn \circ \BP $
where  $\BP \colon \YY \to \XX$
maps the data to the reconstruction space (reconstruction layer or backprojection; no free parameters)
and $\nsn \colon \XX \to \XX$ is a neural network (NN)
whose free parameters are  adjusted to the training data.  In particular, so
called residual networks $\nsn  = \Id_{\XX}  +  \NN$ 
where only the residual part $\NN$ is trained  \cite{he2016deep} showed very accurate  results for solving inverse problems (\cite{antholzer2017deep,chen2017low,jin2017deep,kang2017deep,lee2017deep,majumdar2015real,rivenson2017deep,wang2016accelerating}).
Here and in the following
$\Id_{\XX}$ denotes the identity on $\XX$.
Another class of reconstruction networks learns free parameters
in iterative schemes.
In  such approaches, a sequence of reconstruction networks
$\Ro = \Ro^{(k)}$  is defined by some iterative process
$ \Ro^{(k)} (\data) = \NN_k  (\Phi_k(\signal_{k-1}  ,  \dots, \signal_{0},    \data) )$
where  $\signal_0 $  is some the initial guess,
$\NN_k    \colon \XX \to \XX$ are networks that can be
adjusted to available  training  data,  and  $\Phi_k$  are updates   based on
the data and the previous iterates \cite{adler2017solving,kelly2017deep,kobler2017variational,schlemper2018deep}.

Further existing deep learning approaches for solving inverse  problems 
are based on trained projection operators \cite{chang2017one,gupta2018cnn},
or use  neural networks as trained regularization term \cite{li2018nett}.

While the above deep learning based reconstruction networks
empirically yield  good performance,
none of them is known to be a convergent regularization method.
In this paper we  introduce a new network structure (null space network)
that, when combined with a classical regularization of
the Moore Penrose inverse is shown to provide a convergent
$\M$-regularization method with rates.

\subsection{Proposed null space networks and main results}

As often  argued in the recent literature, deep learning based
reconstruction approaches (especially using  two-stage networks)
lack data consistency, in the sense that outputs of existing reconstruction
networks fail to accurately predict the given data.
In order to overcome this issue, in this
paper, we introduce a new network, that we name  null space network.
The  propose null space network takes the form (see Definition~\ref{def:nsn})
\begin{equation}\label{eq:nsn}
\nsn  = \Id_{\XX} + (\Id_{\XX} -\Ao^{\plus} \Ao ) \NN
\quad \text{ for a network function   $\NN \colon \XX \to \XX$} \,.
\end{equation}
Note that  $\Id_{\XX} -\Ao^{\plus} \Ao = \Po_{\ker(\Ao)}$
equals the  projector onto the null space $\ker(\Ao)$ of $\Ao$.
Consequently, the null space network $\nsn$ satisfies the
property $\Ao \nsn \signal = \Ao \signal$ for all $\signal \in \XX$. This 
yields  that data consistency, which means that
$\Ao \signal = \data$ is invariant among application of a
null space network (compare Figure~\ref{fig:null}).

Suppose $\signal_1, \dots, \signal_N$ are some desired output images
and let $\nsn $ be a trained null space network that approximately maps
$\Ao^{\plus} \Ao \signal_n$  to $\signal_n$.
(See Subsection~\ref{ssec:traing} for a possible training strategy.)
In this paper, we show that
if $(\Bo_\al)_{\al >0}$  is any  classical $\ker(\Ao)^\bot$-regularization,
then   the two-stage  reconstruction network
\begin{equation} \label{eq:two}
\Ro_\al \coloneqq \nsn \circ \Bo_\al \quad \text{ for } \al >0
\end{equation}
yields a  $\M$-regularization  with  $\M \coloneqq  \nsn  \ran(\Ao^{\plus})$.
To the best of our knowledge, these are first  results for regularization by neural networks.
Additionally we will   derive convergence rates  for  $(\Ro_\al)_{\al>0} $
on suitable function classes.

\begin{center}
\begin{figure}[h]
\begin{tikzpicture}
\draw[line width=1.3pt,->] (0,0)--(7,-3);
\node at (8.2,-3.4){$\ran(\Ao^{\plus}) = \ker(\Ao)^\bot$};
\draw[line width=1.3pt,->] (3,-2)--(4.5,1.5);
\node at (4.5,1.8){$\ker(\Ao)$};
\draw [red, line width=1.3pt] plot [smooth] coordinates {(0.5,7/6) (3,1.4) (4,0.5) (5,1) (7,0.7) (8,1)};
\node at (0.5,8.3/6)[above] {$\M \coloneqq \nsn(\ran(\Ao^{\plus}))$};
\filldraw (5,-15/7)circle(1.7pt);
\node at (4.2,-31/14)[below]{$\zsignal_i \coloneqq \Ao^{\plus} \Ao \signal_i$};
\draw [dashed, line width=1.3pt,->] (5,-15/7)--(5+2.9*3/7,-15/7+2.9*1);
\filldraw[color=red] (5+2.95*3/7,-15/7+2.95*1)circle(1.7pt);
\node at (5+10/7,-15/7+3.1*1)[above]{$\nsn(\zsignal_i) = \signal_i$};
\node at (-3,0){};
\end{tikzpicture}
\caption{\label{fig:null} Sketch of the action of a null space network 
$\nsn$  that maps points  $\zsignal_i \in \ran(\Ao^{\plus})$ to 
more desirable elements in  $ \zsignal_i  + \ker(\Ao)$ along the null 
space of $\Ao$.}
\end{figure}
\end{center}

\subsection{Outline}

This paper is organized as follows. In Section~\ref{sec:Mregularization}
we develop a general theory of $\M$-regularization and introduce the notion of
$\M$-generalized inverse (Definition \ref{def:Mgen})  and
$\M$-regularization  methods (Definition \ref{def:Mreg}) generalizing the classical Moore-Penrose generalized
inverse and  regularization concept. We show convergence (see Theorem \ref{thm:conv})
and derive convergence rates (Theorem \ref{thm:rates}) that include regularization via the
null space networks as a special case.
In Section\ref{sec:deep} we introduce the null-space networks, describe  possible
training and extend the  convergence results in the special case
of the null space network (Theorems  \ref{thm:convQ} and  \ref{thm:ratesQ}).
The paper concludes with an outlook presented in Section~\ref{sec:conclusion}.

\section{A theory of $\M$-regularization}
\label{sec:Mregularization}

In this section, we introduce the novel concepts
of $\M$-generalized inverse and $\M$-regularization.
We derive a general class of $\M$-regularization
for which we show convergence and derive
convergence rates.

Throughout this section,  let  $\Ao \colon \XX \to \YY$ be a linear bounded
operator and
$\nun \colon \XX \to  \ker(\Ao) \subseteq \XX$ be Lipschitz  continuous and define
\begin{equation}\label{eq:M}
\M \coloneqq (\Id_\XX +  \nun ) \ran(\Ao^{\plus}) \,.
\end{equation}
The prime example  is
$\nun  = \Po_{\ker(\Ao)} \circ \NN$ being a null space network
with a neural network function $\NN \colon \XX \to \XX$.
This case will be studied in the following section.
The results presented in this section
apply to general Lipschitz  continuous functions $\nun$
whose image is contained in $\ker(\Ao)$.

\subsection{$\M$-regularization methods}

In the following we  denote by $\Ao^{\plus} \colon \dom(\Ao^{\plus}) \subseteq \YY \to  \XX $ the  Moore-Penrose generalized inverse of  $\Ao$, defined  by
 $ \dom(\Ao^{\plus}) \coloneqq \ran\skl{\Ao} \oplus \ran\skl{\Ao}^\bot$ and
\begin{equation}\label{eq:moore}
 \Ao^{\plus} (\data)  \coloneqq \argmin \set{ \snorm{\signal}^2
\mid  \signal \in \XX \wedge \Ao^*\Ao \signal =\Ao^* \data}
\end{equation}
It is well known \cite{engl1996regularization}
that  $\set{\signal \in \XX \mid \Ao^*\Ao \signal =\Ao^* \data} \neq \emptyset$ if
and only  if $ \data \in \ran\skl{\Ao} \oplus \ran\skl{\Ao}^\bot$.
In particular,  $\Ao^{\plus} \data$ is well defined, and can be found as the
unique minimal norm solution of the normal equation $\Ao^*\Ao \signal =\Ao^* \data$

Classical regularization methods aim for  approximating  $\Ao^{\plus} \data$.
In contrast, the null space network will recover  different solutions of the
normal equation. For that purpose we introduce the following
concept.

\begin{definition}[$\M$-generalized inverse]\label{def:Mgen}
We call
$\Ao^\M \colon  \dom(\Ao^{\plus}) \subseteq \YY \to  \XX $ the
$\M$-generalized inverse of $\Ao$ if
\begin{equation} \label{eq:PhiInv}
\forall \data  \in \dom(\Ao^{\plus}) \colon \quad
\Ao^\M \data
= (\Id_{\XX} + \nun ) (\Ao^{\plus} \data) \,.
\end{equation}
\end{definition}

Recall  that  for any $\data \in \dom(\Ao^{\plus})$, the solution set of the
normal equation $\Ao^* \Ao \signal  = \Ao^* \data$ is given by
$\Ao^{\plus} \data + \ker (\Ao)$.
Hence $\Ao^\M \data$ gives a particular   solution of the  normal
equation, that can  be adapted  to a training set.
The $\M$-generalized inverse coincides  with the  Moore-Penrose
generalized inverse   if and   only if $\nun (\signal) =0$ for all
$\signal \in \XX$ in which case $\M = \ker(\Ao)^\bot = \ran(\Ao^{\plus})$.

\begin{lemma}\label{lem:stetig}
The $\M$-generalized inverse is  continuous
if and only if $\ran(\Ao)$ is closed.
\end{lemma}

\begin{proof}
If $\ran(\Ao)$ is closed, then classical results show that $\Ao^{\plus}$ is bounded
(see for example \cite{engl1996regularization}).
Consequently,
$(\Id_{\XX} + \nun )  \circ \Ao^{\plus} $ is bounded too. Conversely, if $\Ao^\M$ is
continuous, then  the identity  $\Po_{\ran(\Ao^{\plus})} \Ao^\M =   \Ao^{\plus}$ implies that 
the Moore-Penrose  generalized inverse   $\Ao^{\plus}$ is bounded 
and  therefore that  $\ran(\Ao)$ is closed.
\end{proof}

Lemma \ref{lem:stetig} shows that as in the case of
the classical  Moore-Penrose generalized inverse,
the $\M$-generalized inverse is discontinuous in the case that
$\ran (\Ao)$ is not closed.    In order to stably solve the equation
$\Ao \signal = \data$   we therefore  require bounded approximations
of the $\M$-generalized inverse. For that purpose, we introduce the
following concept of regularization methods adapted to $\Ao^{\plus}$.

\begin{definition}[$\M$-regularization\label{def:Mreg} method]
Let  $\skl{\Ro_\al}_{\al >0}$ be a
family  of continuous (not necessarily  linear) mappings
$\Ro_\al  \colon \YY \to \XX$ and let $\al^\star \colon \skl{0, \infty} \times \YY
\to \skl{0, \infty}$.
We call the pair  $(\skl{\Ro_\al}_{\al >0}, \alpha^\star)$
a  $\M$-regularization method for the equation $\Ao \signal = \data$
if the following hold:
\begin{itemize}
\item
$ \forall \data \in \YY \colon \lim_{\delta\to 0} \sup \set{ \al^\star\skl{\delta, y^\delta}
\mid
y^\delta \in \YY \wedge \snorm{y^\delta- y} \leq \delta }
 =0$.

 \item
$ \forall \data \in \YY \colon
 \lim_{\delta\to 0} \sup \set{\snorm{ \Ao^\M y -  \Ro_{\al^\star\skl{\delta, y^\delta}} y^\delta} \mid y^\delta \in \YY \wedge \snorm{y^\delta- y} \leq \delta } =0$.
\end{itemize}
In the case that  $(\skl{\Ro_\al}_{\al >0}, \alpha^\star)$ is a $\M$-regularization method for  $\Ao \signal = \data$, then
we call the family $\skl{\Ro_\al}_{\al >0}$ a regularization of $\Ao^\M$
and $\al^\star$ an admissible parameter choice.
\end{definition}

In our generalized notation, a classical regularization method 
for  the equation $\Ao \signal = \data$ corresponds to a   
$\ran(\Ao^{\plus})$-regularization method  for  $\Ao \signal = \data$

\subsection{Convergence analysis}

The following theorem shows that the combination
of a null space network   and a regularization
method  of $\Ao^{\plus}$ yields  a regularization
of $\Ao^\M$.

\begin{theorem}\label{thm:conv}
Suppose $((\Bo_\al)_{\al >0}, \al^\star)$ is any
classical
regularization method for  $\Ao \signal = \data$.
Then, the pair $((\Ro_\al)_{\al >0}, \al^\star)$ with
$\Ro_\al \coloneqq (\Id_{\XX} + \nun )  \circ \Bo_\al$
is  a $\M$-regularization method  for $\Ao \signal = \data$.
In particular, the  family  $(\Ro_\al)_{\al >0}$ is a  regularization
of $\Ao^\M$.
\end{theorem}

\begin{proof}
Because $((\Bo_\al)_{\al >0}, \al^\star)$ is a  
$\ran(\Ao^{\plus})$-regularization method,
it holds that $\lim_{\delta\to 0} \sup \sset{ \al^\star\skl{\delta, y^\delta}
\mid
y^\delta \in \YY \text{ und } \snorm{y^\delta- y} \leq \delta }
 =0$. Let $L$ be a Lipschitz
constant of $\Id_{\XX} + \nun$.
For any $ \data^\delta$ we have
 \begin{multline*}
 \norm{ \Ao^\M y -   \Ro_{\al^\star\skl{\delta, y^\delta}} y^\delta}
=
\norm{  (\Id_{\XX} + \nun ) \circ \Ao^{\plus} y -  (\Id_{\XX} + \nun )\circ \Bo_{\al^\star\skl{\delta, y^\delta}} y^\delta}
\\
\leq
L  \norm{\Ao^{\plus} y -   \Bo_{\al^\star\skl{\delta, y^\delta}} y^\delta} \,.
 \end{multline*}
 Consequently
 \begin{multline*}
 \sup \set{\snorm{ \Ao^\M y -  (\Id_{\XX} + \nun ) \Bo_{\al^\star\skl{\delta, y^\delta}} y^\delta} \mid y^\delta \in \YY \wedge \snorm{y^\delta- y} \leq \delta}
 \\ \leq
 L   \sup \set{\snorm{ \Ao^{\plus} y -    \Bo_{\al^\star\skl{\delta, y^\delta}} y^\delta} \mid y^\delta \in \YY \wedge \snorm{y^\delta- y}\leq \delta}
 \to 0 \,.
 \end{multline*}
 In particular, $( (\Id_{\XX} + \nun ) \circ \Bo_\al)_{\al >0}$ is  a regularization
of $\Ao^\M$.
\end{proof}

A wide class of $\M$-regularization methods can be defined by
a regularizing filter.

\begin{definition}
A family $\kl{g_\alpha}_{\al >0}$  of functions
$g_\al \colon [0,\norm{\Ao^*\Ao}]  \to \R $
 is called a regularizing   filter   if it satisfies
 \begin{itemize}
\item
For all $\al >0$,  $g_\al$ is piecewise continuous;
\item
 $\exists C  >0 \colon  \sup \set{ \sabs{\la g_\al \skl{\la}} \mid \al >0  \wedge
\la \in [0,\la_{\mathrm{max}}]} \leq C$.
\item
$\forall \la \in (0,\norm{\Ao^*\Ao}] \colon
\lim_{\al \to 0} g_\al \skl{\la} = 1/\la$.
\end{itemize}
\end{definition}

\begin{corollary}
Let  $\kl{g_\alpha}_{\al >0}$   be a regularizing   filter  and define
$\Bo_\al \coloneqq
g_{\al} \kl{ \Ao^*\Ao} \Ao^*$.
Then $( (\Id_{\XX} + \nun ) \circ  \Bo_\al)_{\al >0}$ is  a regularization
of $\Ao^\M$.
\end{corollary}

\begin{proof}
The family $(\Bo_\al)_\al$ is a regularization of $\Ao^{\plus}$; see \cite{engl1996regularization}.
Therefore,  according to Theorem~\ref{thm:conv},
 $( (\Id_{\XX} + \nun ) \circ \Bo_\al)_{\al >0}$ is  a regularization  of $\Ao^\M$.
\end{proof}

Basic examples of  filter based regularization methods are
Tikhonov  regularization,  where $g_\al (\la) = 1/(\al+ \la)$,
 and  truncated singular value decomposition where
 \begin{equation*}
 g_\al \kl{ \la } \coloneqq
 \begin{cases}
 0 & \text{ if  } \la < \al \\
 1/\la & \text{ if  } \la \geq \al \,.
\end{cases}
\end{equation*}

Classical  regularization methods are based on approximating  the
Moore-Penrose inverse. In our notation, this corresponds to
a $\ran(\Ao^{\plus})$-regularization methods.
The following  result shows that  $\M$-regularization methods are
essentially continuous approximations of $\Ao^\M$.

\begin{proposition}\label{thm:punktweise}
Let $\skl{\Ro_\al}_{\al >0}$ be a family of continuous
mappings
$\Ro_\al \colon \YY \to \XX$.
\begin{enumerate}[label=(\alph*)]

\item \label{it:punkt1} If $\rest{\Ro_\al}{\dom \skl{\Ao^{\plus}}}
\to \Ao^\M$ pointwise  as $\al \to 0$, then the family $\skl{\Ro_\al}_{\al >0}$ is a  regularization of
$\Ao^\M$.

\item \label{it:punkt2}
Suppose that $\skl{\Ro_\al}_{\al >0}$ is a regularization  of  $\Ao^\M
$ and that there exists a parameter choice  $\al^\star$
that is continuous in the  first argument. Then
$\rest{\Ro_\al}{\dom \skl{\Ao^{\plus}}} \to \Ao^\M$  pointwise as  $\al \to 0$.
\end{enumerate}
\end{proposition}

\begin{proof} \mbox{}
\ref{it:punkt1} If $\rest{\Ro_\al}{\dom \skl{\Ao^{\plus}}} \to \Ao^\M$
pointwise, then $\Po_{\ran(\Ao^{\plus})} \circ \rest{\Ro_\al}{\dom \skl{\Ao^{\plus}}} \to \Po_{\ran(\Ao^{\plus})} \circ   \Ao^\M = \Ao^{\plus}$ pointwise. Hence,  classical regularization theory
implies that   $\Po_{\ran(\Ao^{\plus})} \circ  \Ro_\al$ is a regularization  of $\Ao^{\plus}$.
We have $\Ro_\al =
(\Id_{\XX}+ \nun ) \circ \Po_{\ran(\Ao^{\plus})} \circ \Ro_\al$
and, according  to  Theorem~\ref{thm:conv},
the family $(\Ro_\al)_{\al >0}$ is a regularization of $\Ao^\M$.

\ref{it:punkt2}
We have
\begin{multline*}
\sup \set{\snorm{ \Po_{\ran(\Ao^{\plus})} ( \Ao^\M y -
 \Ro_{\al^\star\skl{\delta, y^\delta}} y^\delta) } \mid y^\delta \in \YY \wedge \snorm{y^\delta- y} \leq \delta }\\
\leq
\sup \set{\snorm{ \Ao^\M y -
\Ro_{\al^\star\skl{\delta, y^\delta}} y^\delta} \mid y^\delta \in \YY \wedge \snorm{y^\delta- y} \leq \delta } \to 0\,,
\end{multline*}
which shows that $(\Po_{\ran(\Ao^{\plus})} \circ \Ro_{\al})_{\al >0}$
is a regularization of $\Ao^{\plus} = \Po_{\ran(\Ao^{\plus})} \circ \Ao^\M$.
Together with   standard regularization theory this shows that
$\Po_{\ran(\Ao^{\plus})} \circ \rest{\Ro_\al}{\dom \skl{\Ao^{\plus}}} \to \Ao^{\plus}$
pointwise as  $\al \to 0$. Consequently, $\rest{\Ro_\al}{\dom \skl{\Ao^{\plus}}}
= (\Id_{\XX}+\nun) \circ \Po_{\ran(\Ao^{\plus})} \circ \rest{\Ro_\al}{\dom \skl{\Ao^{\plus}}}$
converges pointwise to $\Ao^\M = (\Id_{\XX}+\nun) \circ \Ao^{\plus}$.
\end{proof}

\subsection{Convergence rates}

Next we derive   quantitative error estimates.
For that purpose, we assume in the following  that
$\Bo_\al = g_{\al} \kl{ \Ao^*\Ao} \Ao^*$   is
 defined by the regularizing filter $\kl{g_\alpha}_{\al >0}$.
 We use the  notation
$\al^\star  \asymp ({\delta}/{\rho})^{a}$ as $\delta \to 0$
where $\al^\star \colon \YY  \times \skl{0, \infty} \to \skl{0, \infty}$  and 
$a, \rho >0$
to indicate there are positive constants
$d_1, d_2$  such that
$d_1 ({\delta}/{\rho})^{a} \leq \al^\star(\delta)
\leq d_2 ({\delta}/{\rho})^{a}$.

 \begin{theorem} \label{thm:rates}
Suppose $\mu ,\rho >0$ and let $\kl{g_\alpha}_{\al >0}$ be 
a regularizing filter  such that there exist constants $\al_0,c_1, c_2  >0$ 
with  
\begin{itemize}
\item
$\forall  \al >0  \;  \forall \la \in [0, \norm{\Ao^*\Ao}]\colon
\la^\mu \abs{1- \la g_\al \kl{\la}} \leq  c_1 \alpha ^\mu$;

\item
$\forall \al \in( 0, \al_0) \colon  \snorm{g_\al}_\infty \leq c_2/ \al$.
\end{itemize}
Consider the $\M$-regularization method  $\Ro_\al  \coloneqq
(\Id_{\XX} + \nun)   \circ g_\al(\Ao^*\Ao) \Ao^*$ and set
\begin{equation}
\M_{\mu, \rho, \nun} \coloneqq
(\Id_{\XX}+\nun) \kl{\Ao^*\Ao}^\mu \kl{ \overline{B _\rho (0)}}  \,.
\end{equation}
Moreover, let $\al^\star \colon  \skl{0, \infty} \times \YY \to \skl{0, \infty}$
be a parameter choice (possible depending on the source set
$\M_{\mu, \rho, \nun}$) that satisfies
$\al^\star  \asymp ({\delta}/{\rho})^{\frac{2}{2\mu+1}}$ as $\delta \to 0$.
Then there exists a constant $c>0$ such that
\begin{multline}\label{eq:rates}
 \sup \set{\snorm{\Ro_{\al^\star(\delta,y^\delta)} (y^\delta) - \signal } \mid
 \signal \in \M_{\mu, \rho, \nun} \wedge  \data^\delta \in \YY \wedge  \snorm{\Ao \signal - \data^\delta} \leq \delta } \\
\leq c \delta^{\frac{2\mu}{2\mu+1}}
 \rho^{\frac{1}{2\mu+1}} \,.
\end{multline}
In particular, for any
$\signal  \in \ran( (\Id_{\XX}+\nun) \circ (\Ao^*\Ao)^\mu)$
we have the convergence rate result
$\snorm{\Ro_{\al^\star(\delta,y^\delta)} (y^\delta) - \signal } =
\mathcal{O}(\delta^{\frac{2\mu}{2\mu+1}})$.
\end{theorem}

\begin{proof}
We have $\Po_{\ran(\Ao^{\plus})} \M_{\mu, \rho, \nun}
=   \kl{\Ao^*\Ao}^\mu \kl{ \overline{B _\rho (0)}} $
and $\Po_{\ran(\Ao^{\plus})}\Ro_\al  = g_\al(\Ao^*\Ao) \Ao^*$.
Suppose $\signal  \in \M_{\mu, \rho, \nun}$ and
$\data^\delta \in \YY$ with $\snorm{\Ao \signal - \data^\delta} \leq \delta$.
Under the given assumptions, $g_\al(\Ao^*\Ao) \Ao^*$ is an order optimal
 regularization method on   $\kl{\Ao^*\Ao}^\mu \kl{ \overline{B _\rho (0)}} $,
 which  implies (see \cite{engl1996regularization})
 \begin{equation*}
 \snorm{g_{\al^\star(\delta,y^\delta)}(\Ao^*\Ao) \Ao^* (y^\delta)
	- \Po_{\ran(\Ao^{\plus})}  \signal }
	\leq  C \, \delta^{\frac{2\mu}{2\mu+1}}
\end{equation*}
for some  constant $C>0$ independent of 
$\signal$, $\data^\delta$.
Consequently, we have
 \begin{align*}
	\snorm{\Ro_{\al^\star(\delta,y^\delta)} (y^\delta) - \signal }
	&= \snorm{(\Id_{\XX} + \nun)   \circ g_{\al^\star(\delta,y^\delta)}(\Ao^*\Ao) \Ao^* (y^\delta) -
	(\Id_{\XX} + \nun) \Po_{\ran(\Ao^{\plus})}  \signal }
	\\
	&\leq L
	\snorm{g_{\al^\star(\delta,y^\delta)}(\Ao^*\Ao) \Ao^* (y^\delta)
	- \Po_{\ran(\Ao^{\plus})}  \signal }
	\\
	&\leq L C \delta^{\frac{2\mu}{2\mu+1}}
	\rho^{\frac{1}{2\mu+1}} \,,
 \end{align*}
where $L$ is the Lipschitz constant of $\Id_{\XX} + \nun$.
Taking the supremum  over all $\signal  \in \M_{\mu, \rho, \nun}$ and
$\data^\delta \in \YY$ with $\snorm{\Ao \signal - \data^\delta} \leq \delta$
yields  \eqref{eq:rates}.
\end{proof}

Note that the filters $\kl{g_\alpha}_{\al >0}$ of the truncated SVD and the Landweber iteration satisfy the assumptions of Theorem \ref{thm:rates}. In the case of Tikhonov regularization the assumptions are satisfied for $\mu \leq 1$.  In particular under the assumption (resembling the classical  source condition)
\begin{equation*}
\signal \in  (\Id_{\XX}+\nun) ( \ran(\Ao^{\plus}) )
\end{equation*}
we obtain the convergence rate
$\snorm{\Ro_{\al^\star(\delta,y^\delta)} (y^\delta) - \signal } =
\mathcal{O}(\delta^{1/2})$.

\section{Deep null space learning}
\label{sec:deep}

Throughout this section let $\Ao \colon \XX \to \YY$ be a linear bounded
operator.
In this case, we define $\M$-regularizations by
null-space  networks. We describe a possible training strategy
and  derive regularization  properties and rates.
For the following recall that the projector onto the  kernel
of $\Ao$ is given by $\Po_{\ker(\Ao)}  = \Id_{\XX} - \Ao^{\plus} \Ao $.

\subsection{Null  space networks}

For simplicity. we work with layered
feed forward networks, although more complicated networks can be applied as long as their Lipschitz constant is not too large. While their  notation is
standard in a finite-dimensional setting, no formal
definitions seems available  for general Hilbert spaces.
We introduce the following Hilbert space notion.

 \begin{definition}[Layered feed forward network]\label{def:net}
Let $\XX$ and $\ZZ$ be Hilbert spaces.
We call a function $\NN \colon \XX \to \ZZ$
defined by
\begin{equation} \label{eq:nn}
	\NN
	\coloneqq
	\nlo_{L}  \circ \Wo_L \circ \nlo_{L-1} \circ \Wo_{L-1} \circ  \cdots \circ \nlo_1 \circ \Wo_1 \,,
\end{equation}
a layered feed forward neural network function
of depth $L \in \N$ with activations $\nlo_1, \dots, \nlo_L$  if
\begin{enumerate}[label=(N\arabic*)]
\item $\XX_\ell$ are Hilbert spaces with $\XX_0=\XX$ and $\XX_L = \ZZ$;
\item $\Wo_\ell \colon  \XX_{\ell-1} \to \XX_\ell $ are affine, continuous;
\item $\nlo_\ell \colon \XX_\ell \to \XX_\ell $ are continuous.
\end{enumerate}
\end{definition}

Usually the nonlinearities $\nlo_\ell$ are fixed and the affine mappings $\Wo_\ell $
are trained.
In the case that $\XX_\ell$ is a function space, then a standard operation for
$\nlo_\ell$ is the ReLU (the rectified linear unit),  $\operatorname{ReLU} (x) \coloneqq
\max \set{x,0}$, that is applied component-wise, or ReLU in combination with max pooling which takes the maximum value
$\max\set{ \abs{x(i)} \colon i \in I_k}$  within clusters of transform coefficients.
The network in Definition \ref{def:net} may in particular be a convolutional neural
network (CNN);  see \cite{li2018nett} for a definition even in Banach spaces.
In a similar manner one could define  more general feed forward networks in Hilbert spaces,
 for example following the notion of \cite{shalev2014understanding} in the finite dimensional case.

We are now able to formally define the concept of a  null space network.

\begin{definition}\label{def:nsn}
A function $\nsn  \colon \XX \to \XX  $ is a null space network if it has the form
$\nsn = \Id_{\XX} +  (\Id_{\XX} - \Ao^{\plus} \Ao) \NN$ where $\NN \colon \XX \to \XX$ is a neural
network function as in \eqref{eq:nn}.
\end{definition}

\subsection{Network training}
\label{ssec:traing}

We train  the null space  network $\nsn = \Id_{\XX} +  (\Id_{\XX} - \Ao^{\plus} \Ao) \NN$
to (approximately) map  elements to the desired class of training phantoms.
For  that purpose, we fix the following:

\begin{itemize}
\item
$\M_N = \set{\signal_1, \dots, \signal_N}$  is a class of training  phantoms;
\item
For all $\ell$ fix the nonlinearity $\nlo_\ell \colon  \XX_\ell \to \XX_\ell$;
\item
$ \Wset_\ell $ are finite-dimensional spaces  of affine continuous mappings;

\item $\Nset $ is the set of all NN functions  of the form \eqref{eq:nn}
with $ \Wo_\ell \in  \Wset_\ell $.
\end{itemize}

We then consider null space network $\Id_{\XX} +   (\Id_{\XX} - \Ao^{\plus} \Ao)   \NN$
where $\NN \in \Nset $. To train the null space networks we
propose to minimize the regularized error functional
$E \colon \Nset  \to \R$ defined by
\begin{equation} \label{eq:err1}
    E (\NN ) \coloneqq \frac{1}{2} \sum_{\ell=1}^L \norm{ \signal_n - (\Id_{\XX} + (\Id_{\XX} - \Ao^{\plus} \Ao) \NN)  (\Ao^{\plus} \Ao  \signal_n) }^2
    + \mu \prod_{\ell=1}^L \norm{\Lo_\ell}
\end{equation}
where $\NN$ is  of the form \eqref{eq:nn} and $ \Lo_\ell $ is the linear part of $\Wo_\ell$  and
$\mu$ is a regularization parameter.

Network training aims at making $E (\NN )$  small, for example, by gradient descent.
Clearly $\prod_{\ell=1}^L \norm{\Lo_\ell}$  is  an upper bound on the Lipschitz constant
of  $ \NN $.  Therefore, the Lipschitz constant of  the finally trained network will stay reasonably small.
Note that it is not required that  \eqref{eq:err1} is exactly minimized. Any  trained network where
$\frac{1}{2} \sum_{\ell=1}^L \snorm{ \signal_n - (\Id_{\XX} +  \Ao^{\plus} \Ao \NN)  (\Ao^{\plus} \Ao  \signal_n) }^2$ is small
yields a null space network $\Id_{\XX} +  \Ao^{\plus} \Ao \NN$  that does, at least on the training set, a better job in  estimating
$\signal_n$ from $\Ao^{\plus} \Ao  \signal_n$ than the identity.

Alternatively, we may train a regularized null space network $\Id_{\XX} +   ( \Id_{\XX} - \Bo_\al \Ao)   \NN$
to map the  regularized data $\Bo_\al \Ao \signal_n$ (instead of $ \Ao^{\plus} \Ao  \signal_n$)
to the outputs $\signal_n$. This yields
the modified error functional
\begin{equation} \label{eq:err2}
    E_\alpha (\NN ) \coloneqq \frac{1}{2} \sum_{\ell=1}^L \norm{ \signal_n - (\Id_{\XX} + (\Id_{\XX} - \Bo_\al \Ao) \NN)  (\Bo_\al  \signal_n) }^2
    + \mu \prod_{\ell=1}^L \norm{\Lo_\ell} \,.
\end{equation}
Trying to minimize  $E_\alpha$ may be beneficial in the  case that many
singular values are small but do not vanish exactly.
The regularized version $\Bo_\alpha$ might be defined  by
truncated SVD or Tikhonov regularization.

\subsection{Convergence and convergence rates}

Let $\Id_{\XX} +   (\Id_{\XX} - \Ao^{\plus} \Ao) \NN$  be a null-space  network, possibly  trained as  described in
Section \ref{ssec:traing} by approximately minimizing \eqref{eq:err1}.
Any such network  belongs  to the class  of functions $\Id_{\XX} + \nun$ by taking $\nun =  (\Id_{\XX} - \Ao^{\plus} \Ao) \NN$.
Consequently, the convergence theory of Section \ref{sec:Mregularization}
applies. In particular,  Theorem \ref{thm:conv}
shows that  a regularization $(\Bo_\al)_{\al >0}$ of the
Moore-Penrose generalized inverse  defines a
$\M$-regularization  method via $\Ro_\al \coloneqq  (\Id_{\XX} +   (\Id_{\XX} - \Ao^{\plus} \Ao) \NN)  \Bo_\al$.
Additionally, Theorem  \ref{thm:rates} yields convergence rates for the regularization
$(\Ro _\al )_{\al >0}$ of $\Ao^\M$.

In some cases,  the projection $\Po_{\ker(\Ao)} =  \Id_{\XX} - \Ao^{\plus} \Ao$ might be costly to be
computed
exactly. For that purpose, in  this section we derive more general regularization methods 
that include approximate evaluations  of  $\Ao^{\plus} \Ao$.

\begin{theorem} \label{thm:convQ}
Let  $\nsn = \Id_{\XX} + (\Id_{\XX} - \Ao^{\plus} \Ao) \NN $ be a null space network
and set $\M \coloneqq \ran( \nsn )$.
Suppose $((\Bo_\al)_{\al >0}, \al^\star)$ is a
regularization method for  $\Ao \signal = \data$.
Moreover, let $(\Qo_\al)_{\al >0}$ be a family  of bounded
operators  on $\XX$ with $\snorm{\Qo_\al - \Po_{\ker(\Ao)}} \to 0$
as $\al \to 0$.
Then, the pair $((\Ro_\al)_{\al >0}, \al^\star)$ with
\begin{equation}
	\Ro_\al \coloneqq (\Id_{\XX} + \Qo_\al \NN )  \circ \Bo_\al
\end{equation}
is  a $\M$-regularization method  for $\Ao \signal = \data$.
In particular, the  family  $(\Ro_\al)_{\al >0}$ is a  regularization
of $\Ao^\M$.
\end{theorem}

\begin{proof}
We have
\begin{multline} \label{eq:Q}
\norm{(\Id_{\XX} + \Qo_\al \NN  )  \circ \Bo_\al (\data^\delta) - \Ao^\M\data }
\\ \leq
\norm{(\Id_{\XX} + \Po_{\ker(\Ao)} \NN )  \circ \Bo_\al (\data^\delta) - \Ao^\M\data }
+  \snorm{\Qo_\al - \Po_{\ker(\Ao)}} \, \snorm{\NN \Bo_{\al} (\data^\delta)} \,.
\end{multline}
The claim  follows from Theorem \ref{thm:conv}.
\end{proof}

 \begin{theorem} \label{thm:ratesQ}
 Let  $\nsn = \Id_{\XX} + (\Id_{\XX} - \Ao^{\plus} \Ao) \NN $ be a null space network
 and set $\M \coloneqq \ran( \nsn )$.
Let $\mu >0$, suppose  $\kl{g_\alpha}_{\al >0}$
 satisfies  the assumptions of Theorem~\ref{thm:rates}, and
let $(\Qo_\al)_{\al >0}$ be a family  of bounded
operators  on $\XX$ with $\snorm{\Qo_\al - \Po_{\ker(\Ao)}}
=  \mathcal{O}(\delta^{\frac{2\mu}{2\mu+1}})$.
Consider the
regularization $(\Ro_\al)_{\al>0}$ with
\begin{equation}
	\Ro_\al \coloneqq (\Id_{\XX} + \Qo_\al \NN )  \circ g_\al(\Ao^*\Ao) \Ao^* \,.
\end{equation}
Then, the parameter choice
$\al^\star  \asymp ({\delta}/{\rho})^{\frac{2}{2\mu+1}}$
yields the convergence rate results 
$\snorm{\Ro_{\al^\star(\delta, y^\delta)} (y^\delta) - \signal } =
\mathcal{O}(\delta^{\frac{2\mu}{2\mu+1}})$
for any $\signal \in \ran( \nsn (\Ao^*\Ao)^\mu)$.
\end{theorem}

\begin{proof}
Follows from the estimate
\eqref{eq:Q} with $\Bo_\al = g_\al(\Ao^*\Ao) \Ao^* $ and  Theorem \ref{thm:rates}.
\end{proof}

One might use   $\Qo_\al = \Bo_{\phi(\al)}  \Ao $ as a possible
approximation to $\Po_{\ker(\Ao)^\bot} = \Ao^{\plus} \Ao$ for  some  function 
$\phi \colon [0, \infty) \to [0, \infty) $. In such a situation, one can use 
existing software packages  (for example, for the filtered  backprojection  algorithm 
and  the discrete Radon  transform  in case of computed tomography)  
for evaluating   $\Bo_{\phi(\al)}$ and $  \Ao$.

\section{Conclusion}
\label{sec:conclusion}

In this paper, we introduced the concept of
null space networks that have the form $\nsn= \Id_{\XX} +  (\Id_{\XX} - \Ao^{\plus} \Ao) \NN  $,
where $\nun$ is  any neural network function
(for example a  deep convolutional neural network) and $\Id_{\XX} - \Ao^{\plus} \Ao = \Po_{\ker(\Ao)}$ is the
projector onto the kernel of the forward operator $\Ao \colon \XX \to \YY$
of the inverse problem to be solved.
The  null space network
shares similarity with a residual network  that takes the general form
$\Id_{\XX} +  \NN$. However, the introduced projector $\Id_{\XX} - \Ao^{\plus} \Ao$
guarantees data consistency  which is an important issue when
solving inverse problems.

The null space networks are special members of the class
of functions $\Id_{\XX} + \nun$  that satisfy  $\ran(\nun) \subseteq \ker (\Ao)$.
For this class, we introduced the concept of $\M$-generalized
inverse $\Ao^\M$ and $\M$-regularization as point-wise approximations
of $\Ao^\M$ on $\dom(\Ao^{\plus})$. We showed that any classical
regularization $(\Bo_\al)_{\al >0}$ of the Moore-Penrose generalized inverse
defines a  $\M$-regularization method via $(\Id_{\XX} + \nun)  \Bo_\al$.
In the case of null space networks where  $\nun = (\Id_{\XX} - \Ao^{\plus} \Ao) \NN$,
we  additionally derived  convergence results  using only  approximation of the
projection operator  $\Po_{\ker(\Ao)}$. Additionally, we  derived convergence rates using
either exact or approximate projections.

To the best of our knowledge, the obtained
convergence and convergence rates are the first regularization results
for solving inverse problems with neural networks. Future work has to be
done to numerically test the null space networks  for typical
inverse problems  such as limited data problems in CT or deconvolution
and compare the performance with standard  residual networks, iterative
networks or variational networks.

\section*{Acknowledgement}

The work of M.H and S.A. has been supported by the Austrian Science Fund (FWF),
project P 30747-N32.

\end{document}